\documentclass[11pt,twoside,reqno]{amsart}
\allowdisplaybreaks
\usepackage{amsmath,amstext,amssymb,epsfig}
\usepackage{graphicx}
\usepackage{mathrsfs}
\calclayout
\makeatletter
\renewcommand{\@seccntformat}[1]{\bf\csname the#1\endcsname.}
\renewcommand{\section}{\@startsection{section}{1}
	\z@{.7\linespacing\@plus\linespacing}{.5\linespacing}
	{\normalfont\upshape\bfseries\centering}}
\renewcommand{\@biblabel}[1]{\@ifnotempty{#1}{#1.}}
\makeatother
\theoremstyle{plain}
\newtheorem{thm}{Theorem}[section]
\newtheorem{lem}[thm]{Lemma}
\newtheorem{prop}[thm]{Proposition}
\newtheorem{cor}[thm]{Corollary}

\theoremstyle{definition}
\newtheorem{ex}[thm]{Example}
\newtheorem{defn}[thm]{Definition}
\newtheorem{rem}{Remark}[section]

\usepackage{cancel}
\usepackage[parfill]{parskip}
\usepackage[german]{varioref}
\usepackage[all]{xy}
\usepackage{color}

\def\E{{\mathcal E}}
\def\D{{\mathcal D}}

\def \>{\succ}
\def \<{\prec}

\def\D{{\mathcal D}}

\def\E{\mathcal{E}}
\begin{document}	
	\title[Bouzid Mosbahi\textsuperscript{1}, Sania Asif \textsuperscript{2}, Ahmed Zahari\textsuperscript{3*}]{classification of tridendriform algebra and related structures}
	\author{ Bouzid Mosbahi\textsuperscript{1}, Sania Asif \textsuperscript{2*}, Ahmed Zahari\textsuperscript{3*}}
\address{\textsuperscript{1}University of Sfax, Faculty of Sciences of Sfax,  BP 1171, 3000 Sfax, Tunisia.}
   	\address{\textsuperscript{2}School of Mathematics and Statistics, Nanjing University of information science and technology, Nanjing, Jianngsu Province, PR China.}
		\address{\textsuperscript{3}Universit\'{e} de Haute Alsace, IRIMAS-D\'{e}partement de Math\'{e}matiques, 18, rue des Fr\`eres Lumi\`ere F-68093 Mulhouse, France.}
\email{\textsuperscript{1}mosbahibouzid@yahoo.fr}
	\email{\textsuperscript{2*}11835037@zju.edu.cn}
\email{\textsuperscript{3*}zaharymaths@gmail.com}
	
	
	\date{\today}
	\begin{abstract}  The classification of algebraic structures and their derivations is an important and ongoing research area in mathematics and physics, and various results have been obtained in this field. This article presents the classification of tridendriform algebras that was first studied by Loday and Ronco, including an analysis of structure constant equations using computer algebra software. We further explicitly classify the derivations and centroids of tridendriform algebras, showing that there are only trivial derivations for $2$- and $3$-dimensional algebras but $21$ non-isomorphic derivations for $4$-dimensional tridendriform algebras with dimension range from $1$ to $5$. Additionally, for centroids (centroid and quasi-centroid), there are trivial isomorphism classes for $2$ dimensional tridendriform algebra, $6$ non-isomorphic classes for $3$-dimensional tridendriform algebras and $21$ for $4$-dimensional algebras. The dimensions range for centroid is from $1$ to $5$, whereas it is  from $1$ to $10$ for quasi-centroid.
\end{abstract}\footnote{The second and third authors are the corresponding authors.}
\maketitle \section{Introduction}\label{introduction}
 The classification of algebras is an essential topic of research in mathematics that has applications in various fields. Its significance lies in its ability to organize and understand algebraic structures, to simplify the process of proving theorems, and to develop powerful computational methods for studying algebraic structures. The history of classifying algebras can be traced back to the early days of algebraic research  when mathematicians made the first attempts to classify algebras in the $18th$ century, and they studied the properties of polynomials to understand the structure of algebraic systems. Later on, the classification of the algebraic structure continued, and a French mathematician Evariste Galois developed the theory of field extensions in \cite{5}. Classifying simple groups was a significant milestone at that time. The classification of algebras has become an increasingly important topic of research in mathematics for several reasons. One of the main reasons is that the classification of algebras provides a framework for organizing and understanding algebraic structures. For example, in recent years, associative algebras and Hopf algebras have applications in quantum mechanics and quantum field theory see \cite{2,4,3,1}. Moreover, these classifications facilitate us in proving various theorems by reducing the number of cases to consider.
 \par Tridendriform algebra is an algebraic structure that arises in various fields of mathematics and physics. They are closely related to other algebraic structures such as dendriform algebras, associative algebra, $NS-$algebra, and many other operators such as averaging operator and Rota-Baxter operator \cite{6,7}. Due to the significant importance of classifying various algebraic structures in \cite{10,11,8,12} and being the hot topic, we are not only interested in evaluating the complete classification of tridendriform algebra but also interested in the classification of derivation and centroids of tridendriform algbera. A tridendriform algebra is a vector space $V$ over a field $F$ equipped with three bilinear multiplication operations $\<,\>,\vee$ known as right product, left product, and middle product, respectively satisfying particular identities. One of the first significant results in the classification of tridendriform algebras was obtained by Loday and Ronco in 2003. They showed that any finite-dimensional tridendriform algebra over a field of characteristic zero is either isomorphic to a tensor product of classical Lie algebras or the direct sum of a Lie algebra and a tridendriform algebra \cite{9}. Later, various other authors studied the classification of tridendriform algebras, focusing on different aspects. For example, Fiedorowicz and Loday in \cite{13} classified tridendriform algebras with trivial multiplication, and Stanislav Smirnov classified tridendriform algebras with quadratic multiplication. Additionally, various algebras' derivation and centroids concepts are essential growing topics to understand algebraic structures completely. Various authors have studied the derivations of algebras with different perspectives and have obtained valuable results in \cite{14,15,18,17,16}. Moreover, classification results on derivations of various algebras can be seen in \cite{23,20,22,25,21,19,26,24}.
 \par  Despite above mentioned research, no effort has been made to explicitly evaluate all the isomorphism classes of tridendriform algebra and its related structures. Motivated by the above existing research and the need to advance the literature, we investigate the classification of tridendriform algebra of dimension $\leq 4$ over the field of characteristic $0$. And we have evaluated all non-isomorphic derivations and centroids of tridendriform algebra for  the dimension $\leq 4$. We start this paper by presenting the general overview of tridendriform algebra and defining the derivation and centroids structures of tridendriform algebra in Section $2$. Next, in Section $3$, we introduce the new classification method based on the structure of  tridendriform algebra. We then show that any $3$-dimensional tridendriform algebra over a field of characteristic $0$ is isomorphic to one of $6$ possible tridendriform algebras, each characterized by a different structure. Afterwards, we  showed that  $4$-dimensional tridendriform algebra is isomorphic to one of the $21$ possible non-isomorphic tridendriform algebras. Our detailed analysis of the classification of tridendriform algebra  includes the evaluation of structure constant equations and obtaining results by the computer algebra software. Finally, in Section $4$, we explicitly provide the classification of derivations and centroids of tridendriform algebras. This section is divided into three subsections. In the first one, we observe that only trivial derivations exist for $2$ and $3$ dimensional tridendriform algebra. However, for the $4$-dimensional trdendriform algebra, we get $21$ non-isomorphic derivations, with dimensions ranging from $1$ to $8$. In other subsections, we classified centroids of trdendriform algebra. By centroids, we refer to both the centroid and quasi-centroid. We evaluate that centroid, and quasi centroid of $3$ dimensional tridendriform algebra are isomorphic to each $6$ non-isomorphic class with the dimensions ranging from $1$ to $3$ and $3$ to $5$ respectively. Likewise, classification of $4$-dimensional tridendriform algebra and derivation of tridendriform algebras, there exist $21$, non-isomorphic centroids of tridendriform algebra. Moreover, the dimension of centroid and  quasi-centroids are in the range of $1$ to$ 5$ and $1$ to $10 $, respectively.\par
Getting all the classification results by hand is tiresome and requires a lot of focus and energy. A little error in the computation yields wrong information and ultimately provides a poor foundation for mathematics. To avoid such human errors and extraction of reliable results, we performed our computation with the help of a computer and verified the results by hand. In particular, due to the powerful computational capacity to solve abstract algebra problems, we  use mathematical software (Mathematica) to compute our desired classification results. It offers several built-in functions for studying abstract algebraic structures. In particular, we use the Algebra package to get our results. Moreover, classification results obtained from old research are not easy to compare, but paper aims to provide an algorithmic classification, and the results obtained are easily comparable. The article provides a comprehensive and systematic treatment of the $4$-dimensional tridendriform algebras classification over a field of characteristic $0$.
Our results contribute to understanding the structure and properties of tridendriform algebra and pave the way for further research in this area. 
\section{Prelimieries}
\begin{defn}A tridendriform algebra is a quardruple $ (\E,\<,\>,\vee) $  consisting of three  multiplications  maps $\>,\<,\vee : \E \times \E \to \E$, which satisfy the following axioms:
	\begin{equation}
	\begin{aligned}&(a\<b)\<c =a\<(b*c),\\&(a\>b)\<c=a\>(b\<c),\\&(a*b)\>c =a\>(b\>c),\\&(a\>b)\vee c=a\>(b\vee c),\\&(a\<b)\vee c=a \vee (b\>c),\\&(a \vee b)\<c =a\vee(b\< c),\\&(a\vee b)\vee c=a\vee(b\vee c).\end{aligned}
	\end{equation}for $a,b,c \in \E$.
\end{defn}
\begin{defn}
 Let $(\E, \prec, \succ, \vee)$ and $(\E', \succ',\prec', \vee')$ be two tridendriform algebras. A linear map 
$\phi: \E\rightarrow \E'$ is a morphism of tridendriform algebra if the following identities hold
$$\phi \circ\prec=\prec'\circ( \phi\otimes\phi), \quad\phi \circ\succ =\succ'\circ( \phi\otimes\phi)\quad \;\; \mbox{and}\;\;\phi \circ\vee =\vee'\circ( \phi\otimes\phi).$$
\end{defn}
Tridendriform algebra has rich structural properties; we can get many other algebraic structures out of it. Tridendriform algebra also strongly relates to linear operators such as Rota-Baxter operators. As tridendriform algebra involves  three multiplication maps $\<,\>,\vee$, if any one or two maps are zero, we get dendriform and associative algebra, respectively. Moreover, the sum of three multiplication maps can yield an associative algebra structure. 
\begin{prop}
	If $(\E, \<,\>,\vee)$ is a tridendriform algebra, then $(\E, *)$ forms an associative algebra, where $*$ is defined by $a *b := a \> b + a\<b + a \vee b$ for all $ a, b \in \E. $
\end{prop}
\begin{proof} To prove $ (\E, *) $ is an associative algebra, we need to show the following associative identity
$ a* (b *c) = (a*b) * c $. For this, consider that
\begin{eqnarray*}\begin{aligned}&a* (b* c) \\&=a * (b \>  c + b \< c + b \vee c) \\&= a  \>  (b  \>  c + b \<c + b \vee c) + a \< (b\>  c + b \< c + b \vee c) + a \vee (b \> c + b \< c + b \vee c) \\&= a \>  (b \>  c) + a \> (b \<c) + a  \> (b \vee c) + a\< (b \>  c) + a \< (b \<c) + a \<(b \vee c) \\&+ a \vee (b \> c) +  a\vee (b \<c) + a \vee (b \vee c)\\&=(a \< b)\< c + (a \>  b) \< c + (a \< b + a \>  b + a\vee b) \> c
+ (a  \> b)\vee c + (a \vee b) \< c \\&+ (a \< b)\vee c + (a \vee b) \vee c \\&=(a \<b + a  \> b + a \vee b) \< c + (a \<b + a  \> b + a \vee  b) \> c
+ (a  \<b + a  \>  b + a \vee b) \vee c \\&=(a * b) *c.\end{aligned}\end{eqnarray*}This completes the proof.\end{proof}
\begin{defn}
Let $(A,\circ)$ be an associative algebra, a Rota-Baxter operator of weight $\theta$ on $A$ is a linear map $ R: A \to A,$ subjected to the following relation
$$R(a) \circ R(b) =R(R(a)\circ b + a\circ R(b) + \theta (a \circ  b)),$$
for all  $ a, b \in A $ and $\theta\in \mathbf{C}$. 
\end{defn} An algebra with the Rota-Baxter operator defined on it is called a Rota-Baxter algebra.
\begin{prop}
	Let $(\E, \circ, R)$ be an associative Rota-Baxter algebra. Let us define the operation $\>, \<, \vee$ on $ E $ by $ a \<b = a \circ R(b) $, $ a \> b = R(a)\circ b $ and $ a\vee b =\theta (a\circ  b)$, Then $(\E,\<,\>,\vee)$ becomes a  tridendriform algebra.
\end{prop}\begin{proof}
We check the properties of tridendriform algebras one by one, as follows
\begin{enumerate}
 \item\begin{eqnarray*}
	\begin{aligned}
	(a \< b) \< c &= (a\circ R(b)) \circ R(c) = a\circ (R(b) \circ R(c))
	\\&= a \circ (R(R(b) \circ  c + b \circ R(c) + \theta (b \circ c)))
	\\&= a  \< (R(b) \circ c) + a \< (b \circ R(c)) + a \< \theta (b \circ c)\\& = a\< (b \> c) + a\< (b\< c) + a\< (b \vee c).\end{aligned}
	\end{eqnarray*}
	\item\begin{equation*}\begin{aligned}
	(a\>b)\< c &= (R(a) \circ b) \circ R (c) \\&= (R (a))\circ  (b \circ R(c)) \\&= a\>  (b \< c).
	\end{aligned}
	\end{equation*}	
	\item\begin{equation*}\begin{aligned}
	a \>(b \> c) &= R(a)\circ (R(b) \circ c) \\&= (R(a) \circ R(b)) \circ c \\&= R(R(a) \circ b) c + R(a \circ R(b)) c+ R(\theta (a \circ b)) \circ c
	\\&= (a \> b)\> c + (a \< b) \> c+ (a \vee b) \> c.
	\end{aligned}
	\end{equation*}
	\item
	\begin{equation*}(a \> b) \vee c = \theta((R(a) \circ b)\circ c) =\theta(R(a) \circ  (b \circ c)) = a\>  (b \vee c).
	\end{equation*}
	\item
	\begin{equation*}(a \< b) \vee c
	= \theta((a \circ R(b)) \circ c) =\theta(a \circ (R(b) \circ c)) = a\vee (b\>c).
	\end{equation*}
	\item\begin{equation*}(a \vee b) \< c = \theta (a \circ b) \circ R(c) = \theta (a \circ  (b \circ  R(c))) = a \vee (b\< c).
	\end{equation*}
	\item\begin{equation*}(a\vee b)\vee c = \theta^2
	((a \circ  b) \circ c) = \theta^2(a\circ (b \circ c)) = a\vee (b \vee c).
	\end{equation*}
\end{enumerate}
This completes the proof.  
\end{proof}
Let $\E$ be a tridendriform algebra and $A,~B$, be its subsets; then we can define the binary operation $*$ over subsets of $\E$, as follows:
$$A*B=A\>B+A\<B+A\vee B,$$ where
 \begin{center}
     $A\>B=span \{a\> b|~ a\in A, b\in B\}$, \\$ A\<B=span \{a\< b|~ a\in A, b\in B\}$, \\$ A\vee B=span \{a\vee b|~ a\in A, b\in B\}$.
 \end{center}
 Let $A \subseteq \E$ and $A\neq \emptyset$, the subset $$Z_{A}(\E)= \{a\in \E|a\circ S=S\circ a=0\}$$ is called centralizer of $A$ in $\E$.\\ Where the center of tridendriform algebra can be defined as$$Z(\E)= \{a\in \E|~a\circ s=s\circ a,~~ \forall s\in \E\}.$$
\begin{defn}
	Derivation of the tridendriform algebra $\E$ is a linear map $\D:\E\to \E $, such that the following identity holds for all $a,b \in \E$
	$$\D(a\circ b )= \D(a)\circ b + a\circ \D(b).$$ The set of all derivations of $\E$ are denoted by $Der(\E)$.
\end{defn}
\begin{defn}
 A linear map $\D:\E\to \E$ is called central derivation of $\E$, if for $a,b\in \E$, we have $$\D(a\circ b)= \D(a)\circ b= 0.$$ The set of all central derivation of $\E$ are denoted by $ZDer(\E)$.
\end{defn}Moreover, the central derivation is an ideal of the derivation of $\E$. It can be proved as follows.
\begin{prop}
	$ZDer(\E)$ is ideal of $Der(\E)$.
\end{prop}
\begin{proof}To prove, $ZDer(\E)$ is ideal of $Der(\E)$, we first note that $ZDer(\E)$ is a subalgebra of $Der(\E)$. Let us prove the first case; another case can be proved similarly. For $ \D_1\in ZDer(\E)$ and $ \D_2\in Der(\E) $, we have
	\begin{eqnarray*}
	\D_1\D_2(a\circ b)= \D_1(\D_2(a)\circ b + a\circ \D_2(b))= \D_1(\D_2(a)\circ b )+ \D_1(a\circ \D_2(b))=0.
	\end{eqnarray*}
Similarly, for $ \D_2\D_1(a\circ b)=0. $ It completes the proof.\end{proof}
	\begin{defn}
	A linear map $\psi:\E\to \E$ is said to be the centroid of $\E$ if it satisfies the following identity $$\psi(a\circ b)= a\circ \psi(b)= \psi(a)\circ b$$ for all $a,b \in \E$. The set of all centroid of $\E$ is dented by $C(\E)$.
\end{defn} Note that $\circ$ denotes $\<,\>,\vee $ respectively. Moreover centroid of associative algebra is denoted by $$C(A)= C(\E)_\> \cap C(\E)_\< \cap C(\E)_\vee.$$
\begin{defn}
	A map $\psi:\E\to \E$ is considered quasi-centroid of $ \E $ if it satisfies the following identity $$ a\circ \psi(b)= \psi(a)\circ b$$ for all $a,b \in E$. The set of all quasi-centroid of $ \E $ is dented by $QC(\E)$.
\end{defn}
\begin{prop}
	Let $\E$ be a tridendriform algebra. Then we have the following results:\begin{enumerate}
		\item  $[Der(\E), C(\E)]\subseteq C(\E)$
		\item  $[Der(\E), QC(\E)]\subseteq QC(\E)$
	\end{enumerate}
\end{prop}
\begin{proof}
\begin{enumerate}
	\item Suppose that $ \D_1\in Der(\E) $, $ \psi_2\in C(\E) $, for $ a, b \in \E$, we have
	\begin{eqnarray*}
	[\D_1\psi_2(a), b]= \D_1([\psi_2(a), b])-[\psi_2(a), \D_1(b)]= \D_1\psi_2([a, b])-\psi_2[a, \D_1(b)]
	\end{eqnarray*}
	\begin{eqnarray*}
	[\psi_2\D_1(a), b]=\psi_2[\D_1(a), b]= \psi_2(\D_1([a, b])-[a, \D_1(b)])= \psi_2\D_1([a, b])-\psi_2[a, \D_1(b)].
	\end{eqnarray*} So we have $ [[\D_1,\psi_2](a), b]=  [\D_1,\psi_2]([a, b]) $,
	thus $$[\D_1,\psi_2]\in C(\E).$$
	\item Suppose that $ \D_1\in Der(\E) $, $\psi_2\in QC(\E) $, for $ a, b \in \E$, we have
	\begin{eqnarray*}
	[\D_1\psi_2(a), b]= \D_1([\psi_2(a), b])-[\psi_2(a), \D_1(b)]= \D_1([a, \psi_2(b)])-[a, \psi_2\D_1(b)]
	\end{eqnarray*}
	\begin{eqnarray*}
	[\psi_2\D_1(a), b]=[\D_1(a), \psi_2(b)]= \D_1([a, \psi_2(b)])-[a, \D_1\psi_2(b)].
	\end{eqnarray*} So we have $ [[\D_1,\psi_2](a), b]=  [a, [\D_1,\psi_2](b)] $,
	thus $$[\D_1,\psi_2]\in QC(E).$$
\end{enumerate} This completes the proof.
	\end{proof}
\section{Classification of tridendriform algebra}
This section describes the classification of tridendriform algebras ofdimension $\leq 4$ over the field $\mathbb{F}$ of characteristic $0$.
Let $\left\{e_1,e_2, e_3,\cdots, e_n\right\}$ be a basis of an $n$-dimensional tridendriform algebras $\E$. The product of basis  can be expressed in terms of structure constants as follows
$$e_i\prec  e_j=\sum_{k=1}^n\gamma_{ij}^ke_k\quad ;\quad e_i\succ e_j=\sum_{k=1}^n\delta_{ij}^ke_k\quad ;\quad  e_i\vee e_j=\sum_{k=1}^n \xi_{ij}^ke_k.$$ where, the matrix  $(\gamma_{ij}^k)$, $(\delta_{ij}^k)$ and $(\xi_{ij}^k)$ stand for the structure constants of $ \E $.
Now evaluating the structure constant equations for the tridendriform algebras, we get
\begin{eqnarray}
\sum_{p=1}^n(\gamma_{ij}^p\gamma_{pk}^q-\gamma_{jk}^p\gamma_{ip}^q-\delta_{jk}^p\gamma_{ip}^q-\xi_{jk}^p\gamma_{ip}^q)&=&0,\\
\sum_{p=1}^n(\delta_{ij}^p\gamma_{pk}^q-\gamma_{jk}^p\delta_{ip}^q)&=&0,\\
\sum_{p=1}^n(\gamma_{ij}^p\delta_{pk}^q+\delta_{ij}^p\delta_{pk}^q+\xi_{ij}^p\delta_{pk}^q-\delta_{jk}^p\delta_{ip}^q)&=&0,\\
\sum_{p=1}^n(\delta_{ij}^p\xi_{pk}^q-\xi_{jk}^p\delta_{ip}^q)&=&0,\\
\sum_{p=1}^n(\gamma_{ij}^p\xi_{pk}^q-\delta_{jk}^p\xi_{ip}^q)&=&0,\\
\sum_{p=1}^n(\xi_{ij}^p\gamma_{pk}^q-\delta_{jk}^p\xi_{ip}^q)&=&0,\\
\sum_{p=1}^n(\xi_{ij}^p\xi_{pk}^q-\xi_{jk}^p\xi_{ip}^q)&=&0.
\end{eqnarray}
 \begin{ex}
A complex  tridendriform algebra structure on a $2$-dimensional vector space $\E$ with the basis $\left\{e_1, e_2\right\}$ can be obtained by defining the following products.
$$e_2\prec e_2=e_1, \quad e_2\succ e_2=e_1,\quad e_2\vee e_2=e_1.$$
\end{ex}	

\begin{thm}
 Any $ 3 $-dimensional  complex tridendriform algebra either is associative or isomorphic to one of the following pairwise non-isomorphic tridendriform algebra :\\

$\mathcal{DT}_3^1 :
\begin{array}{ll}  
e_1\prec e_1=e_2,\\
e_1\prec e_3=e_2,\\
e_3\prec e_1=e_2,
\end{array}\quad
\begin{array}{ll}  
e_1\succ e_3=e_2,\\
e_3\succ e_1=e_2,\\
e_3\succ e_3=e_2,
\end{array}\quad
\begin{array}{ll}  
e_1\vee e_1=e_2,\\
e_1\vee e_3=e_2,\\
e_3\vee e_3=e_2.
\end{array}$

$\mathcal{DT}_3^2 :
\begin{array}{ll}  
e_1\prec e_1=e_2,\\
e_1\prec e_3=e_2,\\
e_3\prec e_1=e_2,
\end{array}\quad
\begin{array}{ll}  
e_3\prec e_3=e_2,\\
e_3\succ e_1=e_2,\\
e_3\succ e_3=e_2,
\end{array}\quad
\begin{array}{ll}  
e_1\vee e_3=e_2,\\
e_3\vee e_1=e_2,\\
e_3\vee e_3=e_2.
\end{array}$

$\mathcal{DT}_3^3 :
\begin{array}{ll}  
e_1\prec e_2=ae_3,\\
e_2\prec e_1=e_3,\\
e_2\prec e_2=be_3,
\end{array}\quad
\begin{array}{ll}
e_2\succ e_1=ce_3,\\  
e_2\succ e_2=e_3,
\end{array}\quad
\begin{array}{ll}  
e_1\vee e_2=e_3,\\
e_2\vee e_2=de_3.
\end{array}$

$
\mathcal{DT}_3^4 :
\begin{array}{ll}  
e_1\prec e_2=e_3,\\
e_2\prec e_1=e_3,\\
e_2\prec e_2=e_3,
\end{array}\quad
\begin{array}{ll}  
e_1\succ e_2=e_3,\\
e_2\succ e_1=e_3,
\end{array}\quad
\begin{array}{ll}
e_2\succ e_2=e_3,\\  
e_2\vee e_2=e_3.
\end{array}$

$\mathcal{DT}_3^5 :
\begin{array}{ll}  
e_1\prec e_1=e_3,\\
e_1\prec e_2=e_3,\\
e_2\prec e_1=e_3,
\end{array}\quad
\begin{array}{ll}  
e_1\succ e_1=e_3,\\
e_2\succ e_1=e_3,\\
e_2\succ e_2=e_3,
\end{array}\quad
\begin{array}{ll}
e_2\vee e_1=e_3,\\ 
e_2\vee e_2=e_3.
\end{array}$

$\mathcal{DT}_3^6 :
\begin{array}{ll}  
e_1\prec e_1=e_3,\\
e_1\prec e_2=ae_3,\\
e_2\prec e_1=e_3,
\end{array}\quad
\begin{array}{ll} 
e_2\prec e_2=-e_3,\\ 
e_1\succ e_1=be_3,\\
e_1\succ e_2=ce_3,
\end{array}\quad
\begin{array}{ll}
e_2\succ e_2=e_3,\\
e_1\vee e_1=-ce_3,\\ 
e_2\vee e_2=e_3.
\end{array}$

\end{thm}
\begin{proof}
We give the proof of one case.
Suppose that tridendriform algebra $\mathcal{A}=(TDend, \prec)$ has the following 
multiplication table : 
$$\begin{array}{ll}
e_1 \prec e_1=e_3,\quad 
e_1 \prec e_2=e_3,\quad 
e_2 \prec e_1=e_3.
\end{array}$$
We define $\mathcal{B}=(TDend,\succ,\vee)$ by the multiplication table : 

$$\begin{array}{ll}  
e_1 \succ e_1 = a_1e_1+a_2e_2+a_3e_3,\\
e_1 \succ e_2= a_4e_1+a_5e_2+a_6e_3,\\
e_1 \succ e_3 = a_7e_1+a_8e_2+a_9e_3,\\
e_2 \succ e_1 = b_{1}e_1+b_{2}e_2+b_{3}e_3,\\
e_2 \succ e_2 = b_{4}e_1+b_{5}e_2+b_{6}e_3,\\
e_2 \succ e_3 = b_{7}e_1+b_{8}e_2+b_{9}e_3,\\
e_3\succ e_1 = c_{1}e_1+c_{2}e_2+c_{3}e_3,\\
e_3\succ e_2 = c_{4}e_1+c_{5}e_2+c_{6}e_3,\\
e_3\succ e_3 = c_{7}e_1+c_{8}e_2+c_{9}e_3,\\
\end{array}\quad\quad
\begin{array}{ll}  
e_1 \vee e_1 = x_1e_1+x_2e_2+x_3e_3,\\
e_1 \vee e_2= x_4e_1+x_5e_2+x_6e_3,\\
e_1 \vee e_3 = x_7e_1+x_8e_2+x_9e_3,\\
e_2 \vee e_1 = y_{1}e_1+y_{2}e_2+y_{3}e_3,\\
e_2 \vee e_2 = y_{4}e_1+y_{5}e_2+y_{6}e_3,\\
e_2 \vee e_3 = y_{7}e_1+y_{8}e_2+y_{9}e_3,\\
e_3\vee e_1 = z_{1}e_1+z_{2}e_2+z_{3}e_3,\\
e_3\vee e_2 = z_{4}e_1+z_{5}e_2+z_{6}e_3,\\
e_3\vee e_3 = z_{7}e_1+z_{8}e_2+z_{9}e_3,
\end{array}\quad$$

where $a_i,b_i,c_i,x_i,y_i$ and $z_i$
are unknowns $(i =1,2,...,9).$
Verifying the tridendriform algebras axioms, we get the following 
constraints for the structure  :
\begin{align*}
a_1=a_2=a_4=a_5=a_6=a_7=a_8=a_9=0,\\
b_1=b_2=b_4=b_5=b_6=b_7=b_8=b_9=0,\\
c_1=c_2=c_3=c_4=c_5=c_6=c_7=c_8=c_9=0,\\
x_1=x_2=x_3=x_4=x_5=x_6=x_7=x_8=x_9=0,\\
y_1=y_2=y_4=y_5=y_7=y_8=y_9=0,\\
z_1=z_2=z_3=z_4=z_5=z_6=z_7=z_8=z_9=0,
\end{align*}
and
\begin{align*}
a_3=b_3=b_6=y_3=y_6=1.
\end{align*}
We find the table of multiplication as follows :

$$\begin{array}{ll}  
e_1 \prec e_1 = e_3,\\
e_1 \prec e_2= e_3,\\
e_2 \prec e_1= e_3,
\end{array}\quad\quad
\begin{array}{ll}  
e_1 \succ e_1= e_3,\\
e_2 \succ e_1 = e_3,\\
e_2 \succ e_2 = e_3,
\end{array}\quad\quad
\begin{array}{ll}  
e_2 \vee e_1 = e_3,\\
e_2 \vee e_2 = e_3.
\end{array}$$

This is the tridendriform algebra $\mathcal{DT}_3^5$.
Similar observations can be applied for the other cases. 
\end{proof}
\begin{thm}
 Any 4-dimensional complex tridendriform algebra either is associative or isomorphic to one of the following pairwise non-isomorphic tridendriform algebra : \\

$\mathcal{DT}_4^1 :
\begin{array}{ll}  
e_1\prec e_1=e_1,\\
e_1\succ e_2=e_2,
\end{array}\quad
\begin{array}{ll}  
e_3\vee e_1=e_3+e_4,\\
e_3\vee e_3=e_3+e_4,
\end{array}\quad
\begin{array}{ll}  
e_3\vee e_4=e_3+e_4,\\
e_4\vee e_4=e_3+e_4.
\end{array}$\\

$\mathcal{DT}_4^2 :
\begin{array}{ll}  
e_2\prec e_1=e_3,\\
e_4\succ e_1=e_3
\end{array}\quad 
\begin{array}{ll}  
e_1\vee e_2=e_3+e_4,\\
e_2\vee e_1=e_2,
\end{array}\quad e_4\vee e_4=e_3+e_4.$

$
\mathcal{DT}_4^3 :
\begin{array}{ll}  
e_1\prec e_2=e_3+ae_4,\\
e_2\prec e_1=e_3+e_4,\\
e_2\prec e_2=be_3-ae_4,
\end{array}\quad
\begin{array}{ll}  
e_1\succ e_2=e_3,\\
e_2\succ e_1=be_3,\\
e_2\succ e_2=e_3,
\end{array}\quad
\begin{array}{ll}  
e_1\vee e_2=e_4,\\
e_2\vee e_1=e_4,\\
e_2\vee e_2=de_4.
\end{array}$

$\mathcal{DT}_4^4 :
\begin{array}{ll}  
e_1\prec e_2=e_3+e_4,\\
e_2\prec e_1=e_3+e_4,\\
e_2\prec e_2=e_3+e_4,
\end{array}\quad
\begin{array}{ll}  
e_1\succ e_2=e_3,\\
e_2\succ e_1=e_3,\\
e_2\succ e_2=e_3,
\end{array}\quad
\begin{array}{ll}  
e_1\vee e_2=e_3+e_4,\\
e_2\vee e_1=e_3+e_4,\\
e_2\vee e_2=e_3+e_4.
\end{array}$

$
\mathcal{DT}_4^5 :
\begin{array}{ll}  
e_1\prec e_3=e_2+e_4,\\
e_3\prec e_1=e_2+e_4,
\end{array}\quad
\begin{array}{ll} 
e_3\succ e_1=e_2+e_4,\\
e_3\succ e_3=e_2+e_4,
\end{array}\quad
\begin{array}{ll} 
e_3\vee e_3=e_2+e_4.
\end{array}$

$
\mathcal{DT}_4^6 :
\begin{array}{ll}  
e_1\prec e_3=e_2+e_4,\\
e_3\prec e_1=e_2+e_4,\\
e_3\prec e_3=e_2+e_4,\\
\end{array}\quad
\begin{array}{ll} 
e_1\succ e_3=e_2+e_4,\\
e_3\succ e_3=e_2+e_4,
\end{array}\quad
\begin{array}{ll} 
e_3\vee e_1=e_2+e_4,\\
e_3\vee e_3=e_2+e_4.
\end{array}$

$\mathcal{DT}_4^7 :
\begin{array}{ll}  
e_1\prec e_3=e_2+e_4,\\
e_3\prec e_1=e_2+e_4,\\
e_3\prec e_3=e_2+e_4,\\
\end{array}\quad
\begin{array}{ll} 
e_1\succ e_3=e_2+e_4,\\
e_3\succ e_1=e_2,\\
e_3\succ e_3=e_2+e_4,\\
\end{array}\quad
\begin{array}{ll} 
e_1\vee e_3=e_2+e_4,\\
e_3\vee e_1=e_2+e_4,\\
e_3\vee e_3=e_4.
\end{array}$

$\mathcal{DT}_4^8 :
\begin{array}{ll}  
e_1\prec e_1=e_2,\\
e_1\prec e_3=e_2,
\end{array}\quad
\begin{array}{ll} 
e_3\prec e_1=e_2,\\
e_1\succ e_3=e_2,\
\end{array}\quad
\begin{array}{ll}  
e_3\succ e_1=e_2,\\
e_1\vee e_1=e_2,
\end{array}\quad
\begin{array}{ll} 
e_3\vee e_1=e_2,\\
e_3\vee e_3=e_2.
\end{array}$

$\mathcal{DT}_4^{9} :
\begin{array}{ll}  
e_1\prec e_1=e_2,\\
e_1\prec e_3=ae_2,\\
e_3\prec e_3=-ae_2,\\
\end{array}\quad
\begin{array}{ll} 
e_4\prec e_4=e_2,\\
e_3\succ e_1=be_2,\\ 
e_3\succ e_3=e_2,
\end{array}\quad
\begin{array}{ll} 
e_4\succ e_4=ce_2,\\
e_1\vee e_1=e_2,\\ 
e_4\vee e_4=de_2.
\end{array}$

$\mathcal{DT}_4^{10} :
\begin{array}{ll}  
e_1\prec e_3=e_2,\\
e_3\prec e_1=e_2,\\
e_4\prec e_4=e_2,\\
\end{array}\quad
\begin{array}{ll} 
e_3\succ e_1=e_2,\\
e_3\succ e_3=e_2,\\ 
e_4\succ e_4=e_2,
\end{array}\quad
\begin{array}{ll} 
e_1\vee e_1=e_2,\\
e_3\vee e_3=e_2,\\ 
e_4\vee e_4=e_2.
\end{array}$

$\mathcal{DT}_4^{11} :
\begin{array}{ll}  
e_1\prec e_2=e_3,\\
e_2\prec e_1=e_3,\\
e_2\prec e_2=e_3+e_4,\\
\end{array}\quad
\begin{array}{ll} 
e_1\succ e_2=e_3,\\
e_2\succ e_1=e_3,\\ 
e_2\succ e_2=e_3+e_4,
\end{array}\quad
\begin{array}{ll} 
e_\vee e_2=e_3,\\
e_2\vee e_1=e_3,\\ 
e_2\vee e_2=e_3+e_4.
\end{array}$

$\mathcal{DT}_4^{12} :
\begin{array}{ll}  
e_1\prec e_1=e_4,\\
e_2\prec e_1=e_4,\\
e_2\prec e_3=e_4,
\end{array}\quad
\begin{array}{ll} 
e_1\succ e_2=e_4,\\
e_2\succ e_2=e_4,\\
e_3\succ e_2=e_4,
\end{array}\quad
\begin{array}{ll} 
e_1\vee e_1=e_4,\\
e_2\vee e_2=e_4,\\
e_3\vee e_3=e_4.
\end{array}$

$\mathcal{DT}_4^{13} :
\begin{array}{ll}  
e_2\prec e_2=e_4,\\
e_2\prec e_3=e_4,\\
e_3\prec e_1=e_4,
\end{array}\quad
\begin{array}{ll} 
e_1\succ e_2=e_4,\\
e_3\succ e_1=e_4,\\
e_3\succ e_2=e_4,
\end{array}\quad
\begin{array}{ll} 
e_2\vee e_3=e_4,\\
e_3\vee e_1=e_4,\\
e_3\vee e_3=e_4.
\end{array}$

$\mathcal{DT}_4^{14} :
\begin{array}{ll}  
e_2\prec e_3=ae_4,\\
e_3\prec e_2=e_4,\\
e_3\prec e_3=-ae_4,
\end{array}\quad
\begin{array}{ll} 
e_3\succ e_1=be_4,\\
e_3\succ e_2=e_4,\\
e_3\succ e_3=e_4,
\end{array}\quad
\begin{array}{ll} 
e_1\vee e_1=ae_4,\\
e_3\vee e_1=e_4,\\
e_3\vee e_3=-ae_4.
\end{array}$

$\mathcal{DT}_4^{15} :
\begin{array}{ll}  
e_1\prec e_2=e_4,\\
e_3\prec e_2=e_4,\\
e_3\prec e_3=e_4,
\end{array}\quad
\begin{array}{ll} 
e_1\succ e_1=e_4,\\
e_2\succ e_3=e_4,\\
e_3\succ e_1=e_4,
\end{array}\quad
\begin{array}{ll} 
e_2\vee e_2=e_4,\\
e_2\vee e_3=e_4,\\
e_3\vee e_3=e_4.
\end{array}$

$\mathcal{DT}_4^{16} :
\begin{array}{ll}  
e_2\prec e_2=e_1,\\
e_2\prec e_3=e_1,\\
e_2\prec e_4=e_1,
\end{array}\quad
\begin{array}{ll} 
e_2\succ e_2=e_1,\\
e_3\succ e_2=e_1,\\
e_3\succ e_4=e_1,
\end{array}\quad
\begin{array}{ll} 
e_2\vee e_2=e_1,\\
e_3\vee e_3=e_1,\\
e_4\vee e_2=e_1.
\end{array}$

$\mathcal{DT}_4^{17} :
\begin{array}{ll}  
e_2\prec e_3=e_1,\\
e_3\prec e_3=ae_1,\\
e_4\prec e_3=be_1,
\end{array}\quad
\begin{array}{ll} 
e_2\succ e_4=-ae_1,\\
e_4\succ e_3=ae_1,\\
e_4\succ e_4=e_1,
\end{array}\quad
\begin{array}{ll} 
e_2\vee e_4=e_1,\\
e_3\vee e_3=e_1,\\
e_4\vee e_4=ae_1.
\end{array}$

$\mathcal{DT}_4^{18} :
\begin{array}{ll}  
e_3\prec e_4=e_1,\\
e_4\prec e_2=e_1,\\
e_4\prec e_4=e_1,
\end{array}\quad
\begin{array}{ll} 
e_3\succ e_3=e_1,\\
e_3\succ e_4=e_1,\\
e_4\succ e_4=e_1,
\end{array}\quad
\begin{array}{ll} 
e_3\vee e_2=e_1,\\
e_3\vee e_3=e_1,\\
e_4\vee e_3=e_1.
\end{array}$

$\mathcal{DT}_4^{19} :
\begin{array}{ll}  
e_2\prec e_2=e_1,\\
e_3\prec e_2=e_1,\\
e_4\prec e_3=e_1,
\end{array}\quad
\begin{array}{ll}
e_4\prec e_4=e_1,\\ 
e_2\succ e_4=e_1,\\
e_3\succ e_3=e_1,\\
e_4\succ e_3=e_1,
\end{array}\quad
\begin{array}{ll}
e_4\succ e_4=e_1,\\ 
e_3\vee e_3=e_1,\\
e_4\vee e_4=e_1.
\end{array}$

$\mathcal{DT}_4^{20} :
\begin{array}{ll}  
e_1\prec e_3=e_2,\\
e_3\prec e_1=a^2e_4,\\
e_3\prec e_3=e_2,
\end{array}\quad
\begin{array}{ll}
e_1\succ e_3=e_2,\\
e_3\succ e_1=e_4,\\
e_3\succ e_3=ae_4,
\end{array}\quad
\begin{array}{ll} 
e_1\vee e_3=be_4,\\
e_3\vee e_1=e_4,\\
e_3\vee e_3=e_2.
\end{array}$

$\mathcal{DT}_4^{21} :
\begin{array}{ll}  
e_1\prec e_3=e_2,\\
e_3\prec e_1=e_2,\\
e_3\prec e_3=e_2,\\
\end{array}\quad
\begin{array}{ll} 
e_1\succ e_3=e_2+e_4,\\
e_3\succ e_1=e_2+e_4,\\
e_3\succ e_3=e_2+e_4,\\
\end{array}\quad
\begin{array}{ll} 
e_1\vee e_3=e_4,\\
e_3\vee e_1=e_4,\\
e_3\vee e_3=e_4.
\end{array}$
\end{thm}
\begin{proof}

We give the proof  of only  one case. Other cases can be proved similarly. Suppose that a tridendriform algebra $\mathcal{A}=(TDend, \prec)$ has the following multiplication: 
$$\begin{array}{ll}
e_3 \prec e_4=e_1,\quad 
e_4 \prec e_2=e_1,\quad 
e_4 \prec e_4=e_1.
\end{array}$$

We define $\mathcal{B}=(Tridend,\succ,\vee)$ by the multiplication: 

$$\begin{array}{ll}  
e_1 \succ e_1 = a_1e_1+a_2e_2+a_3e_3+a_4e_4,\\
e_1 \succ e_2 = a_5e_1+a_6e_2+a_7e_3+a_8e_4,\\
e_1 \succ e_3 = a_9e_1+a_{10}e_2+a_{11}e_3+a_{12}e_4,\\
e_1 \succ e_4 = a_{13}e_1+a_{14}e_2+a_{15}e_3+a_{16}e_4,\\
e_2 \succ e_1 = b_1e_1+b_2e_2+b_3e_3+b_4e_4,\\
e_2 \succ e_2 = b_5e_1+b_6e_2+b_7e_3+b_8e_4,\\
e_2 \succ e_3 = b_9e_1+b_{10}e_2+b_{11}e_3+b_{12}e_4,\\
e_2 \succ e_4 = b_{13}e_1+b_{14}e_2+b_{15}e_3+b_{16}e_4,\\
e_3 \succ e_1 = r_1e_1+r_2e_2+r_3e_3+r_4e_4,\\
e_3 \succ e_2 = r_5e_1+r_6e_2+r_7e_3+r_8e_4,\\
e_3 \succ e_3 = r_9e_1+r_{10}e_2+r_{11}e_3+r_{12}e_4,\\
e_3 \succ e_4 = r_{13}e_1+r_{14}e_2+r_{15}e_3+r_{16}e_4,\\
e_4 \succ e_1 = s_1e_1+s_2e_2+s_3e_3+s_4e_4,\\
e_4 \succ e_2 = s_5e_1+s_6e_2+s_7e_3+s_8e_4,\\
e_4 \succ e_3 = s_9e_1+s_{10}e_2+s_{11}e_3+s_{12}e_4,\\
e_4 \succ e_4 = s_{13}e_1+s_{14}e_2+s_{15}e_3+s_{16}e_4,\\
\end{array}\quad\quad
\begin{array}{ll} 
e_1 \vee e_1 = x_1e_1+x_2e_2+x_3e_3+x_4e_4,\\
e_1 \vee e_2 = x_5e_1+x_6e_2+x_7e_3+x_8e_4,\\
e_1 \vee e_3 = x_9e_1+x_{10}e_2+x_{11}e_3+x_{12}e_4,\\
e_1 \vee e_4 = x_{13}e_1+x_{14}e_2+x_{15}e_3+x_{16}e_4,\\
e_2 \vee e_1 = y_1e_1+y_2e_2+y_3e_3+y_4e_4,\\
e_2 \vee e_2 = y_5e_1+y_6e_2+y_7e_3+y_8e_4,\\
e_2 \vee e_3 = y_9e_1+y_{10}e_2+y_{11}e_3+y_{12}e_4,\\
e_2 \vee e_4 = y_{13}e_1+y_{14}e_2+y_{15}e_3+y_{16}e_4,\\
e_3 \vee e_1 = z_1e_1+z_2e_2+z_3e_3+z_4e_4,\\
e_3 \vee e_2 = z_5e_1+z_6e_2+z_7e_3+z_8e_4,\\
e_3 \vee e_3 = z_9e_1+z_{10}e_2+z_{11}e_3+z_{12}e_4,\\
e_3 \vee e_4 = z_{13}e_1+z_{14}e_2+z_{15}e_3+z_{16}e_4,\\
e_4 \vee e_1 = t_1e_1+t_2e_2+t_3e_3+t_4e_4,\\
e_4 \vee e_2 = t_5e_1+t_6e_2+t_7e_3+t_8e_4,\\
e_4 \vee e_3 = t_9e_1+t_{10}e_2+t_{11}e_3+t_{12}e_4,\\
e_4 \vee e_4 = t_{13}e_1+t_{14}e_2+t_{15}e_3+t_{16}e_4.
\end{array}$$

where $a_i,b_i,r_i,s_i,x_i,y_i,z_i$ and $t_i$
are unknowns $(i =1,2,...,16).$
Verifying the tridendriform algebras axioms, we get the following constraints
for the structure constants:
\begin{align*}
a_1=a_2=a_3=a_4=a_5=a_6=a_7=a_8=a_9=a_{10}=a_{11}=a_{12}=a_{13}=a_{14}=a_{15}=a_{16}=0,\\
b_1=b_2=b_3=b_4=b_5=b_6=b_7=b_8=b_9=b_{10}=b_{11}=b_{12}=b_{13}=b_{14}=b_{15}=b_{16}=0,\\
r_1=r_2=r_3=r_4=r_5=r_6=r_7=r_8=0,\\
s_1=s_2=s_3=s_4=s_5=s_6=s_7=s_8=s_9=s_{10}=s_{11}=s_{12}=0,\\
x_1=x_2=x_3=x_4=x_5=x_6=x_7=x_8=x_9=x_{10}=x_{11}=x_{12}=x_{13}=x_{14}=x_{15}=x_{16}=0,\\
y_1=y_2=y_3=y_4=y_5=y_6=y_7=y_8=y_9=y_{10}=y_{11}=y_{12}=y_{13}=y_{14}=y_{15}=y_{16}=0,\\
z_1=z_2=z_3=z_4=z_{13}=z_{14}=z_{15}=z_{16}=0,\\
t_1=t_2=t_3=t_4=t_5=t_6=t_7=t_8=t_{13}=t_{14}=t_{15}=t_{16}=0,
\end{align*}
and
\begin{align*}
r_9=r_{10}=r_{11}=r_{12}=r_{13}=r_{14}=r_{15}=r_{16}=1,\\
s_{13}=s_{14}=s_{15}=s_{16}=1,\\
z_5=z_6=z_7=z_8=z_9=z_{10}=z_{11}=z_{12}=1,\\ 
  t_9=t_{10}=t_{11}=t_{12}=1.
\end{align*}
We find the table of multiplication as follows:

$$\begin{array}{ll}  
e_3 \prec e_4=e_1,\\
e_4 \prec e_2=e_1\\
e_4 \prec e_4= e_1
\end{array}\quad\quad
\begin{array}{ll} 
e_3 \succ e_3= e_1,\\
e_3 \succ e_4= e_1,\\
e_4 \succ e_4= e_1,
\end{array}\quad\quad
\begin{array}{ll} 
e_3 \vee e_2= e_1,\\
e_3 \vee e_3= e_1,\\
e_4 \vee e_3= e_1.
\end{array}$$

This is the tridendriform algebra $\mathcal{DT}_4^{18}$.
Similar observations can be applied to other cases. 
\end{proof}
\section{Classification of derivations and centroids of tridendriform algebra}
\subsection{Classification of derivations}
This section describes the classification of derivation algebras of $2$- and  $3$-dimensional tridendriform algebras over the field $\mathbb{F}$ of characteristic $0$. 
Let $\left\{e_1,e_2, e_3,\cdots, e_n\right\}$ be a basis of an $n$-dimensional tridendriform algebras $\mathcal{A}.$ The product of basis can be expressed in terms of structure constants as follows :
\begin{eqnarray}
\mathcal{D}(e_i)=\sum_{j=1}^nd_{ji}e_j\nonumber.
\end{eqnarray}
We have 
\begin{eqnarray}
\sum_{k=1}^nd_{pk}\gamma_{ij}^k&=&\sum_{k=1}^nd_{ip}\gamma_{pj}^k+\sum_{k=1}^nd_{pj}\gamma_{ip}^k,\label{eqd1}\\
\sum_{k=1}^nd_{pk}\delta_{ij}^k&=&\sum_{k=1}^nd_{ip}\delta_{pj}^k+\sum_{k=1}^nd_{pj}\delta_{ip}^k,\label{eqd2}\\
\sum_{k=1}^nd_{pk}\xi_{ij}^k&=&\sum_{k=1}^nd_{ip}\xi_{pj}^k+\sum_{k=1}^nd_{pj}\xi_{ip}^k.\label{eqd3}
\end{eqnarray}

\begin{thm}\label{theoC1}
The \textbf{derivation} of $4$-dimensional tridendriform algebras 
has the following form:\\

\begin{tabular}{||c||c||c||c||c||c||c||c||c||c||c||c||}
\hline
IC&Der$(\mathcal{DT}_m^n)$&$Dim(Der)$&IC&Der$(\mathcal{DT}_m^n)$&$Dim(Der)$\\
			\hline
$\mathcal{DT}_4^3$&
$\left(\begin{array}{cccc}
d_{11}&0&0&0\\
0&d_{11}&0&0\\
0&0&d_{33}&2k_1\\
0&0&0&k_2
\end{array}
\right)$
&
3
&
$\mathcal{DT}_4^{4}$&
$\left(\begin{array}{cccc}
d_{11}&0&0&0\\
0&d_{11}&0&0\\
0&0&d_{33}&k_1\\
0&0&0&k_2
\end{array}
\right)$
&
4
\\ \hline
$\mathcal{DT}_4^5$&
$\left(\begin{array}{cccc}
0&0&0&0\\
0&d_{22}&0&-d_{22}\\
0&0&0&0\\
0&d_{42}&0&-d_{42}
\end{array}
\right)$
&
2
&
$\mathcal{DT}_4^{6}$&
$\left(\begin{array}{cccc}
0&0&0&0\\
0&d_{22}&0&-d_{22}\\
0&0&0&0\\
0&d_{42}&0&-d_{42}
\end{array}
\right)$
&
2
\\ \hline
$\mathcal{DT}_4^7$&
$\left(\begin{array}{cccc}
d_{11}&0&0&0\\
0&d_{22}&0&0\\
0&0&d_{11}&k_3\\
0&d_{42}&0&k_4
\end{array}
\right)$
&
5
&
$\mathcal{DT}_4^{8}$&
$\left(\begin{array}{cccc}
d_{11}&0&0&0\\
0&d_{22}&0&k_1\\
0&0&0&0\\
0&d_{42}&0&d_{44}
\end{array}
\right)$
&
4
\\ \hline
$\mathcal{DT}_4^9$&
$\left(\begin{array}{cccc}
d_{11}&0&0&0\\
0&2d_{11}&0&0\\
0&0&d_{11}&0\\
0&0&0&d_{11}
\end{array}
\right)$
&
1
&
$\mathcal{DT}_4^{10}$&
$\left(\begin{array}{cccc}
d_{11}&0&0&0\\
0&2d_{11}&0&0\\
0&0&d_{11}&0\\
0&0&0&d_{11}
\end{array}
\right)$
&
1
\\ \hline
$\mathcal{DT}_4^{11}$&
$\left(\begin{array}{cccc}
d_{11}&0&0&0\\
0&2d_{11}&0&0\\
0&0&d_{33}&k_{3}\\
0&0&d_{43}&k_{4}
\end{array}
\right)$
&
4
&
$\mathcal{DT}_4^{13}$&
$\left(\begin{array}{cccc}
d_{11}&0&0&0\\
0&d_{11}&0&0\\
0&0&d_{11}&0\\
0&0&0&2d_{11}
\end{array}
\right)$
&
1
\\ \hline
$\mathcal{DT}_4^{14}$&
$\left(\begin{array}{cccc}
d_{11}&0&0&0\\
0&d_{11}&0&0\\
0&0&d_{11}&0\\
0&0&0&2d_{11}
\end{array}
\right)$
&
1
&
$\mathcal{DT}_4^{15}$&
$\left(\begin{array}{cccc}
d_{11}&0&0&0\\
0&d_{11}&0&0\\
0&0&d_{11}&0\\
0&0&0&2d_{11}
\end{array}
\right)$
&
1
\\ \hline
\end{tabular}

\begin{tabular}{||c||c||c||c||c||c||c||c||c||c||c||c||}
\hline
IC&Der$(\E)$ &$Dim(\D)$&IC&Der$(\E)$&$Dim(\D)$\\
			\hline
$\mathcal{DT}_4^{16}$&
$\left(\begin{array}{cccc}
\frac{d_{11}}{2}&0&0&0\\
0&\frac{d_{11}}{2}&0&0\\
0&0&\frac{d_{11}}{2}&0\\
0&0&0&\frac{d_{11}}{2}
\end{array}
\right)$
&
1
&
$\mathcal{DT}_4^{17}$&
$\left(\begin{array}{cccc}
d_{11}&0&0&0\\
d_{21}&\frac{d_{11}}{2}&0&0\\
0&0&\frac{d_{11}}{2}&0\\
0&0&0&\frac{d_{11}}{2}
\end{array}
\right)$
&
2
\\ \hline
$\mathcal{DT}_4^{18}$&
$\left(\begin{array}{ccccc}
d_{11}&d_{12}&0&0\\
0&\frac{d_{11}+d_{12}}{2}&0&0\\
0&0&\frac{d_{11}+d_{12}}{2}&0\\
0&0&0&\frac{d_{11}+d_{12}}{2}
\end{array}
\right)$
&
2
&
$\mathcal{DT}_4^{19}$&
$\left(\begin{array}{cccc}
d_{11}&0&0&0\\
0&\frac{d_{11}}{2}&0&0\\
0&0&\frac{d_{11}}{2}&0\\
0&0&0&\frac{d_{11}}{2}
\end{array}
\right)$
&
1
\\ \hline
$\mathcal{DT}_4^{20}$&
$\left(\begin{array}{cccc}
d_{11}&0&0&0\\
0&2d_{11}&0&0\\
0&0&d_{33}&2k_{3}\\
0&0&d_{43}&k_{4}
\end{array}
\right)$
&
4
&
$\mathcal{DT}_4^{21}$&
$\left(\begin{array}{cccc}
d_{11}&0&0&0\\
0&2d_{11}&0&0\\
0&0&d_{33}&k_{3}\\
0&0&d_{43}&k_{4}
\end{array}
\right)$
&
4
\\ \hline
\end{tabular}
\end{thm}

	\begin{proof}
We provide the proof of only one case, other cases 
can be addressed similarly with or without
modification(s). Let's consider $\mathcal{DT}_4^{9}$. 
the systems of equations (\ref{eqd1}), (\ref{eqd2}) and 
(\ref{eqd3}) we get
\begin{center}$d_{12}=d_{13}=d_{14}=d_{21}=d_{23}=d_{24}=d_{31}=d_{32}=d_{34}=d_{41}=d_{42}=d_{43}=0,$
$ d_{22}=2d_{11},\quad d_{33}=d_{11}\quad d_{44}=d_{11}.$\end{center}
Hence, the derivations of $\mathcal{DT}_4^{9}$ are indicated as follows\\
$$d_1=\left(\begin{array}{cccc}
1&0&0&0\\
0&2&0&0\\
0&0&1&0\\
0&0&0&1
\end{array}
\right)$$
is the basis of $Der(\mathcal{DT}_4^{9})$ and $Dim(Der(\mathcal{DT}_4^{9}))=1.$ 
	\end{proof}

\vspace{.2in}
\begin{cor}\,
\begin{itemize}
	\item There exist only trivial derivations for the $3$-dimensional tridendriform algebras.
	\item For the $4$-dimensional tridendriform algebras, the dimension of derivation algebra ranges from one to eight.
\end{itemize}
\end{cor}
\begin{rem}
\noindent In the above-displayed tables, the following notations are used :
\begin{itemize}
	\item $k_1=2d_{11}-d_{33},\quad k_{2}=d_{11}-d_{43},\quad k_{3}=2d_{11}-d_{22},\quad k_{4}=2d_{11}-d_{42}.$
	\item IC: Class isomorphism.
\end{itemize}
\end{rem}
\subsection{Classification of Centroid}
\begin{defn}
Let $(\E, \prec,\succ, \vee)$ be a tridendriform algebra. A linear map
 $\psi : \mathcal{E}\rightarrow \mathcal{E}$ is called an element of \textbf{centroid} on $\mathcal{E}$, if for all $x, y\in \mathcal{E}$
\begin{eqnarray}
\psi(x)\prec y&=&\psi(x\prec y)=x\prec \psi(y),\\ 
 \psi(x)\succ y&=&\psi(x\succ  y)=x\succ \psi(y),\\
 \psi(x)\vee y&=&\psi(x\vee y)=x\vee\psi(y).
\end{eqnarray}
The set of all  elements of \textbf{centroid} of $\mathcal{E}$ is denoted $C(\mathcal{E})$.
 \end{defn}
This section describes the centroid  of tridendriform algebras for dimensions $2$ and $3$ over the field $\mathbb{F}$ of characteristic $0$.
Let $\left\{e_1,e_2, e_3,\cdots, e_n\right\}$ be a basis of an $n$-dimensional tridendriform algebras $\mathcal{E}.$ The product of basis is expressed in terms of structure constants as follows
\begin{eqnarray}
\psi(e_i)=\sum_{j=1}^na_{ji}e_j\nonumber.
\end{eqnarray}
By using the above definition of $\psi$ and basis multiplication of tridendriform algebra, 
we get the following equations of structure constants :
\begin{eqnarray}
\sum_{p=1}^na_{pi}\gamma_{pj}^q&=&\sum_{p=1}^n\gamma_{ij}^pa_{qp}=\sum_{p=1}^na_{pj}
\gamma_{ip}^q,\label{eqc1}\\
\sum_{p=1}^na_{pi}\delta_{pj}^q&=&\sum_{p=1}^n\delta_{ij}^pa_{qp}=
\sum_{p=1}^na_{pi}\delta_{ip}^q,\label{eqc2}\\
\sum_{p=1}^na_{pi}\xi_{pj}^q&=&\sum_{p=1}^n\xi_{ij}^pa_{qp}=\sum_{p=1}^na_{pj}\xi_{ip}^q\label{eqc3}
\end{eqnarray}

\begin{thm}\label{theoc1}
The \textbf{centroid} of $3$-dimensional tridendriform algebras has the following form 
:\\

\begin{tabular}{||c||c||c||c||c||c||c||c||c||c||c||c||}
\hline
IC&$C(\mathcal{DT}_m^n)$ &$Dim(C)$&IC&$C(\mathcal{DT}_m^n)$&$Dim(C)$\\
			\hline
$\mathcal{DT}_3^1$&
$\left(\begin{array}{cccc}
a_{11}&0&0\\
0&a_{11}&0\\
0&0&a_{11}
\end{array}
\right)$
&
1
&
$\mathcal{DT}_3^{2}$&
$\left(\begin{array}{cccc}
a_{11}&0&0\\
0&a_{11}&0\\
0&0&a_{11}
\end{array}
\right)$
&
1
\\ \hline
$\mathcal{DT}_3^3$&
$\left(\begin{array}{cccc}
a_{33}&0&0\\
0&a_{33}&0\\
a_{31}&a_{32}&a_{33}
\end{array}
\right)$
&
3
&
$\mathcal{DT}_3^{4}$&
$\left(\begin{array}{cccc}
a_{11}&0&0\\
0&a_{11}&0\\
a_{31}&a_{32}&a_{11}
\end{array}
\right)$
&
3
\\ \hline
$\mathcal{DT}_3^5$&
$\left(\begin{array}{cccc}
a_{11}&0&0\\
0&a_{11}&0\\
a_{31}&a_{32}&a_{11}
\end{array}
\right)$
&
3
&
$\mathcal{DT}_3^{6}$&
$\left(\begin{array}{cccc}
a_{33}&0&0\\
0&a_{33}&0\\
a_{31}&a_{32}&a_{33}
\end{array}
\right)$
&
3
\\ \hline
\end{tabular}\end{thm}
    \begin{proof}
 To illustrate the used approach, we provide the proof of one case. The other cases can be proved similarly with or without modification(s). Let's consider 
$\mathcal{DT}_3^{1}$. Applying the systems of equations 
(\ref{eqc1}), (\ref{eqc2}) and (\ref{eqc3}) 
we get
$a_{12}=a_{13}=a_{21}=a_{23}=a_{31}=a_{32}=0,\quad a_{22}=a_{11},\quad a_{33}=2a_{11}.$
Thus, the centroids of $\mathcal{DT}_3^{1}$ are reported as follows\\
$$a_1=\left(\begin{array}{ccc}
1&0&0\\
0&1&0\\
0&0&1
\end{array}
\right)$$
is the basis of $C(\mathcal{DT}_3^1)$ and $Dim(C)=1.$ This completes the proof.
	\end{proof}
 \newpage
	\begin{thm}\label{theoc2}
The \textbf{centroid} of $4$-dimensional tridendriform algebras has the following form:\\

\begin{tabular}{||c||c||c||c||c||c||c||c||c||c||c||c||}
\hline
IC&$C(\mathcal{DT}_m^n)$ &$Dim(C)$&IC&$C(\mathcal{DT}_m^n)$&$Dim(C)$\\
			\hline
$\mathcal{DT}_4^1$&
$\left(\begin{array}{cccc}
a_{11}&0&0&0\\
0&a_{11}&0&0\\
0&0&a_{11}&0\\
0&0&0&a_{11}
\end{array}
\right)$
&
1
&
$\mathcal{DT}_4^{2}$&
$\left(\begin{array}{cccc}
a_{11}&0&0&0\\
0&a_{11}&0&0\\
a_{31}&0&a_{11}&0\\
0&0&0&a_{11}
\end{array}
\right)$
&
2
\\ \hline
$\mathcal{DT}_4^3$&
$\left(\begin{array}{cccc}
a_{22}&0&0&0\\
0&a_{22}&0&0\\
a_{31}&a_{32}&a_{22}&0\\
0&0&0&a_{22}\\
\end{array}
\right)$
&
3
&
$\mathcal{DT}_4^{4}$&
$\left(\begin{array}{cccc}
a_{33}&0&0&0\\
0&a_{33}&0&0\\
a_{31}&a_{32}&a_{33}&0\\
0&0&0&a_{33}\\
\end{array}
\right)$
&
3
\\ \hline
$\mathcal{DT}_4^5$&
$\left(\begin{array}{cccc}
a_{22}&0&0&0\\
0&a_{22}&0&0\\
0&0&a_{22}&0\\
0&0&0&a_{22}\\
\end{array}
\right)$
&
1
&
$\mathcal{DT}_4^{6}$&
$\left(\begin{array}{cccc}
a_{33}&0&0&0\\
0&a_{33}&0&0\\
0&0&a_{33}&0\\
0&0&0&a_{33}\\
\end{array}
\right)$
&
1
\\ \hline
$\mathcal{DT}_4^7$&
$\left(\begin{array}{cccc}
a_{22}&0&0&0\\
0&a_{22}&0&0\\
0&0&a_{22}&0\\
0&0&0&a_{22}\\
\end{array}
\right)$
&
1
&
$\mathcal{DT}_4^{8}$&
$\left(\begin{array}{cccc}
a_{22}&0&0&0\\
0&a_{22}&0&0\\
0&0&a_{22}&0\\
0&0&0&a_{22}\\
\end{array}
\right)$
&
1
\\ \hline
$\mathcal{DT}_4^9$&
$\left(\begin{array}{cccc}
a_{22}&0&0&0\\
0&a_{22}&0&0\\
0&0&a_{22}&0\\
0&0&0&a_{22}
\end{array}
\right)$
&
1
&
$\mathcal{DT}_4^{10}$&
$\left(\begin{array}{cccc}
a_{22}&0&0&0\\
0&a_{22}&0&0\\
0&0&a_{22}&0\\
0&0&0&a_{22}
\end{array}
\right)$
&
1
\\ \hline
$\mathcal{DT}_4^{11}$&
$\left(\begin{array}{cccc}
a_{22}&0&0&0\\
0&a_{22}&0&0\\
a_{31}&a_{32}&a_{22}&0\\
0&0&0&a_{22}
\end{array}
\right)$
&
1
&
$\mathcal{DT}_4^{12}$&
$\left(\begin{array}{cccc}
a_{44}&0&0&0\\
0&a_{44}&0&0\\
0&0&a_{44}&0\\
a_{41}&&a_{43}&a_{44}
\end{array}
\right)$
&
3
\\ \hline
$\mathcal{DT}_4^{13}$&
$\left(\begin{array}{cccc}
a_{44}&0&0&0\\
0&a_{44}&0&0\\
0&0&a_{44}&0\\
a_{41}&&a_{43}&a_{44}
\end{array}
\right)$
&
3
&

$\mathcal{DT}_4^{14}$&
$\left(\begin{array}{cccc}
a_{44}&0&0&0\\
0&a_{44}&0&0\\
0&0&a_{44}&0\\
a_{41}&&a_{43}&a_{44}
\end{array}
\right)$
&
3
\\ \hline
$\mathcal{DT}_4^{15}$&
$\left(\begin{array}{cccc}
a_{11}&0&0&0\\
0&a_{11}&0&0\\
0&0&a_{11}&0\\
0&a_{42}&a_{43}&a_{11}
\end{array}
\right)$
&
3
&
$\mathcal{DT}_4^{16}$&
$\left(\begin{array}{cccc}
a_{11}&a_{12}&a_{13}&a_{14}\\
0&a_{11}&0&0\\
0&0&a_{11}&0\\
0&0&0&a_{11}
\end{array}
\right)$
&
4
\\ \hline
$\mathcal{DT}_4^{17}$&
$\left(\begin{array}{cccc}
a_{11}&a_{12}&a_{13}&a_{14}\\
0&a_{11}&0&0\\
0&0&a_{11}&0\\
0&0&0&a_{11}
\end{array}
\right)$
&
4
&
$\mathcal{DT}_4^{18}$&
$\left(\begin{array}{cccc}
a_{11}&a_{12}&a_{13}&a_{14}\\
0&a_{11}&0&0\\
0&0&a_{11}&0\\
0&0&0&a_{11}
\end{array}
\right)$
&
4
\\ \hline
$\mathcal{DT}_4^{19}$&
$\left(\begin{array}{cccc}
a_{11}&a_{12}&a_{13}&a_{14}\\
0&a_{11}&0&0\\
0&0&a_{11}&0\\
0&0&0&a_{11}
\end{array}
\right)$
&
4
&

$\mathcal{DT}_4^{20}$&
$\left(\begin{array}{cccc}
a_{44}&0&0&0\\
a_{21}&a_{44}&a_{23}&0\\
0&0&a_{44}&0\\
a_{41}&0&a_{43}&a_{44}
\end{array}
\right)$
&
5
\\ \hline
$\mathcal{DT}_4^{21}$&
$\left(\begin{array}{cccc}
a_{11}&0&0&0\\
a_{21}&a_{11}&c_{23}&0\\
a_{31}&a_{32}&a_{11}&0\\
a_{41}&0&c_{43}&a_{11}
\end{array}
\right)$
&
5
&
&

&

\\ \hline
\end{tabular}
\end{thm}

	\begin{proof}
Let us consider $\mathcal{DT}_4^{1}$. Applying the systems of equations 
(\ref{eqc1}), (\ref{eqc2}) and (\ref{eqc3})  we get
$a_{21}=a_{31}=a_{41}=a_{42}=a_{21}=a_{31}=a_{41}=a_{42}=0,$
$ a_{22}=a_{11},\quad a_{33}=a_{11}\quad a_{44}=a_{11}.$
As a result, the centroids of $\mathcal{DT}_4^{1}$ are identified as follows\\
$$a_1=\left(\begin{array}{cccc}
1&0&0&0\\
0&1&0&0\\
0&0&1&0\\
0&0&0&1
\end{array}
\right)$$
is the basis of $C(\mathcal{DT}_4^{1}))$ and $Dim(C)=1$. The centroid of the remaining parts of dimension $3$  tridendriform algebras can be handled in a
similar manner as depicted above.
	\end{proof}
\begin{cor}\,
\begin{itemize}
	\item The dimensions of the centroid of $3$-dimensional  tridendriform algebras range between $1$ and $3$.
	\item The dimensions of the centroid of $4$-dimensional tridendriform algebras range between $1$ and $5$.
\end{itemize}
\end{cor}
\subsection{Classification of Quasi-Centroid}
\begin{defn}
Let $(\E, \prec,\succ, \vee)$ be a  tridendriform algebra. A linear map
 $\psi : \mathcal{E}\rightarrow \mathcal{E}$ is called an element of \textbf{Quasi-Centroids} on $\mathcal{E}$, if for all $x, y\in \mathcal{E}$
\begin{eqnarray}
\psi(x)\prec y&=&x\prec \psi(y),\\ 
 \psi(x)\succ y&=&x\succ \psi(y),\\
 \psi(x)\vee y&=&x\vee\psi(y).
\end{eqnarray}
The \textbf{ Quasi-centroid} of $\mathcal{E}$ is denoted $QC(\mathcal{E})$. 
 \end{defn}
This section describes the quasi-centroid of tridendriform algebra with dimension $\leq 4$ over the field $\mathbb{F}.$ 
Let $\left\{e_1,e_2, e_3,\cdots, e_n\right\}$ be a basis of an $n$-dimensional tridendriform algebras $\mathcal{\E}.$ The product of basis is expressed in terms of structure constants as follows
\begin{eqnarray}
\psi(e_i)=\sum_{j=1}^na_{ji}e_j\nonumber.
\end{eqnarray}
By using the above definition of $\psi$ and basis multiplication of tridendriform algebra, we get the following equations of structure constants :

\begin{eqnarray}\label{eqQC}
\sum_{p=1}^na_{pi}\gamma_{pj}^q=\sum_{p=1}^na_{pj}\gamma_{ip}^q,\quad 
\sum_{p=1}^na_{pi}\delta_{pj}^q=\sum_{p=1}^na_{pj}\delta_{ip}^q,\quad
\sum_{p=1}^na_{pi}\xi_{pj}^q=\sum_{p=1}^na_{pj}\xi_{ip}^q.
\end{eqnarray}

\begin{thm}\label{theoQC}
The \textbf{Quasi-centroid} of 
$3$-dimensional tridendriform algebras 
has the following form :\\

\begin{tabular}{||c||c||c||c||c||c||c||c||c||c||c||c||}
\hline
IC&$QC(\mathcal{DT}_m^n)$ &$Dim(QC)$&IC&$QC(\mathcal{DT}_m^n)$&$Dim(QC)$\\
			\hline
$\mathcal{DT}_3^1$&
$\left(\begin{array}{cccc}
a_{11}&0&0\\
0&a_{11}&0\\
0&0&a_{11}
\end{array}
\right)$
&
3
&
$\mathcal{DT}_3^{2}$&
$\left(\begin{array}{cccc}
a_{11}&0&0\\
a_{21}&a_{22}&a_{23}\\
0&0&a_{11}
\end{array}
\right)$
&
4
\\ \hline
\end{tabular}

\begin{tabular}{||c||c||c||c||c||c||c||c||c||c||c||c||}
\hline
IC&$C(\mathcal{DT}_m^n)$ &$Dim(C)$&IC&$C(\mathcal{DT}_m^n)$&$Dim(C)$\\
			\hline
$\mathcal{DT}_3^3$&
$\left(\begin{array}{cccc}
a_{11}&0&0\\
0&a_{11}&0\\
a_{31}&a_{32}&a_{33}
\end{array}
\right)$
&
4
&
$\mathcal{DT}_3^{4}$&
$\left(\begin{array}{cccc}
a_{11}&a_{21}&0\\
0&a_{11}&0\\
a_{31}&a_{32}&a_{33}
\end{array}
\right)$
&
5
\\ \hline
$\mathcal{DT}_3^5$&
$\left(\begin{array}{cccc}
a_{11}&0&0\\
0&a_{11}&0\\
a_{31}&a_{32}&a_{33}
\end{array}
\right)$
&
4
&
$\mathcal{DT}_3^{6}$&
$\left(\begin{array}{cccc}
a_{11}&0&0\\
0&a_{11}&0\\
a_{31}&a_{32}&a_{33}
\end{array}
\right)$
&
4
\\ \hline
\end{tabular}
\end{thm}

	\begin{proof}
Let us consider $\mathcal{DT}_3^{1}$.
Applying the systems of equations  (\ref{eqQC}), we get
$a_{12}=a_{13}=a_{21}=a_{23}=a_{31}=a_{32}=0$ and $a_{22}=a_{33}$. 
Hence, the quasi-centroids of $\mathcal{DT}_3^{1}$ are indicated as follows\\
$$a_1=\left(\begin{array}{cccc}
1&0&0\\
0&1&0\\
0&0&1
\end{array}
\right)$$
is the basis of $Der(QC)$ and  Dim$Der(QC)=1.$ The quoisi-centroids of the remaining parts of dimension three tridendriform algebras can be handled in a
similar manner as illustrated above.
\end{proof}
\begin{thm}\label{theoQC2}
The \textbf{Quasi-centroid}  of $4$-dimensional tridendriform algebras have the 
following form:\\

\begin{tabular}{||c||c||c||c||c||c||c||c||c||c||c||c||}
\hline
IC&Der$(QC)$ &$Dim(QC)$&IC&Der$(QC)$&$Dim(QC)$\\
			\hline
$\mathcal{DT}_4^1$&
$\left(\begin{array}{cccc}
a_{11}&0&0&0\\
0&a_{11}&0&0\\
0&0&a_{11}&0\\
0&0&0&a_{11}
\end{array}
\right)$
&
1
&
$\mathcal{DT}_4^{2}$&
$\left(\begin{array}{cccc}
a_{11}&a_{12}&0&0\\
0&0&0&0\\
a_{31}&a_{32}&a_{33}&a_{34}\\
0&0&0&0
\end{array}
\right)$
&
6
\\ \hline
$\mathcal{DT}_4^3$&
$\left(\begin{array}{cccc}
a_{11}&0&0&0\\
a_{21}&a_{11}+a_{21}&0&0\\
a_{31}&a_{32}&a_{33}&a_{34}\\
a_{41}&a_{42}&a_{43}&a_{44}\\
\end{array}
\right)$
&
10
&
$\mathcal{DT}_4^{4}$&
$\left(\begin{array}{cccc}
a_{11}&a_{21}&0&0\\
a_{21}&a_{11}+a_{21}&0&0\\
a_{31}&a_{32}&a_{33}&a_{34}\\
a_{41}&a_{42}&a_{43}&a_{44}\\
\end{array}
\right)$
&
10
\\ \hline
$\mathcal{DT}_4^5$&
$\left(\begin{array}{cccc}
a_{11}&0&0&0\\
a_{21}&a_{22}&a_{23}&a_{24}\\
0&0&a_{11}&0\\
a_{41}&a_{42}&a_{43}&a_{44}\\
\end{array}
\right)$
&
9
&
$\mathcal{DT}_4^{6}$&
$\left(\begin{array}{cccc}
a_{11}&0&0&0\\
a_{21}&a_{22}&a_{23}&a_{24}\\
0&0&a_{11}&0\\
a_{41}&a_{42}&a_{43}&a_{44}\\
\end{array}
\right)$
&
9
\\ \hline
$\mathcal{DT}_4^7$&
$\left(\begin{array}{cccc}
a_{11}&0&0&0\\
a_{21}&a_{22}&a_{23}&a_{24}\\
0&0&a_{11}&0\\
a_{41}&a_{42}&a_{43}&a_{44}\\
\end{array}
\right)$
&
9
&
$\mathcal{DT}_4^{8}$&
$\left(\begin{array}{cccc}
a_{11}&0&0&0\\
a_{21}&a_{22}&a_{23}&a_{24}\\
0&0&a_{11}&0\\
a_{41}&a_{42}&a_{43}&a_{44}\\
\end{array}
\right)$
&
9
\\ \hline
$\mathcal{DT}_4^9$&
$\left(\begin{array}{cccc}
a_{11}&0&0&0\\
a_{21}&2a_{22}&a_{23}&a_{24}\\
0&0&a_{11}&0\\
0&a_{42}&a_{43}&a_{44}\\
\end{array}
\right)$
&
8
&
$\mathcal{DT}_4^{10}$&
$\left(\begin{array}{cccc}
a_{11}&0&0&0\\
a_{21}&a_{22}&a_{23}&a_{24}\\
0&0&a_{11}&0\\
0&a_{42}&a_{43}&a_{44}\\
\end{array}
\right)$
&
8
\\ \hline
$\mathcal{DT}_4^{11}$&
$\left(\begin{array}{cccc}
a_{11}&a_{12}&0&0\\
a_{21}&a_{11}&a_{23}&a_{24}\\
a_{31}&0&a_{33}&0\\
0&a_{42}&a_{43}&a_{44}\\
\end{array}
\right)$
&
10
&
$\mathcal{DT}_4^{12}$&
$\left(\begin{array}{cccc}
a_{11}&0&0&0\\
0&a_{11}&0&0\\
0&0&a_{11}&0\\
a_{41}&a_{42}&a_{43}&a_{44}
\end{array}
\right)$
&
5
\\ \hline
\end{tabular}

\begin{tabular}{||c||c||c||c||c||c||c||c||c||c||c||c||}
\hline
IC&Der$(QC)$ &$Dim(QC)$&IC&Der$(QC)$&$Dim(QC)$\\
			\hline
$\mathcal{DT}_4^{13}$&
$\left(\begin{array}{cccc}
a_{11}&0&0&0\\
0&a_{11}&0&0\\
0&0&a_{11}&0\\
a_{41}&a_{42}&a_{43}&a_{44}
\end{array}
\right)$
&
5
&
$\mathcal{DT}_4^{14}$&
$\left(\begin{array}{cccc}
a_{11}&0&0&0\\
0&a_{11}&0&0\\
0&0&a_{11}&0\\
a_{41}&a_{42}&a_{43}&a_{44}
\end{array}
\right)$
&
5
\\ \hline
$\mathcal{DT}_4^{15}$&
$\left(\begin{array}{cccc}
a_{11}&0&0&0\\
0&a_{11}&0&0\\
0&0&a_{11}&0\\
a_{41}&a_{42}&a_{43}&a_{44}
\end{array}
\right)$
&
5
&
$\mathcal{DT}_4^{16}$&
$\left(\begin{array}{cccc}
a_{11}&a_{12}&a_{13}&a_{14}\\
0&a_{22}&a_{22}&a_{22}\\
0&0&0&0\\
0&0&0&0
\end{array}
\right)$
&
5
\\ \hline
$\mathcal{DT}_4^{17}$&
$\left(\begin{array}{cccc}
a_{11}&a_{12}&a_{13}&a_{14}\\
0&a_{22}&a_{22}&a_{22}\\
0&0&0&0\\
0&0&0&0
\end{array}
\right)$
&
5
&
$\mathcal{DT}_4^{18}$&
$\left(\begin{array}{cccc}
a_{11}&a_{12}&a_{13}&a_{14}\\
0&a_{22}&a_{22}&a_{22}\\
0&0&0&0\\
0&0&0&0
\end{array}
\right)$
&
5
\\ \hline
$\mathcal{DT}_4^{19}$&
$\left(\begin{array}{cccc}
a_{11}&a_{12}&a_{13}&a_{14}\\
0&a_{22}&a_{22}&a_{22}\\
0&0&0&0\\
0&0&0&0
\end{array}
\right)$
&
5
&
$\mathcal{DT}_4^{20}$&
$\left(\begin{array}{cccc}
0&0&0&0\\
a_{21}&a_{22}&a_{23}&a_{24}\\
0&0&0&0\\
a_{41}&a_{42}&a_{43}&a_{44}\\
\end{array}
\right)$
&
8
\\ \hline
$\mathcal{DT}_4^{21}$&
$\left(\begin{array}{cccc}
0&0&0&0\\
a_{21}&a_{22}&a_{23}&a_{24}\\
0&0&0&0\\
a_{31}&a_{32}&a_{33}&a_{34}\\
0&0&0&0
\end{array}
\right)$
&
8
&
&
&
\\ \hline
\end{tabular}
\end{thm}
	\begin{proof}	
Consider that $\mathcal{DT}_4^{1}$.
Applying the systems of equations (\ref{eqQC}), we get
$a_{12}=a_{13}=a_{14}=a_{21}=a_{23}=a_{24}=a_{31}=a_{32}=a_{34}=a_{41}=a_{42}=a_{43}=0 $ and $a_{11}=a_{22}=a_{33}=a_{44}$. Hence, the quasi-centroids of $\mathcal{DT}_4^{1}$ are indicated as follows\\
$$a_1=\left(\begin{array}{ccccc}
1&0&0&0\\
0&1&0&0\\
0&0&1&0\\
0&0&0&1
\end{array}
\right)$$
is the basis of $ QC(\mathcal{DT}_4^{1})$ and $Dim(QC)=1.$ The quasi-centroid of the remaining parts of dimension four tridendriform algebras can be handled similarly, as illustrated above. \end{proof}
\begin{cor}\,
\begin{itemize}
	\item The dimensions of the quasi-centroid of $3$-dimensional tridendriform algebras range between $3$ and $5$.
	\item The dimensions of the quasi-centroid of $4$-dimensional tridendriform algebras range between $1$ and $10$.
	\end{itemize}
\end{cor}

\end{document}